\documentclass[11pt]{amsart}

\usepackage[utf8]{inputenc}
\usepackage[margin=3.1cm]{geometry}
\usepackage{xcolor}
\usepackage{graphicx}
\usepackage{todonotes}
\usepackage{subcaption}
\usepackage{mathtools}

\mathtoolsset{showonlyrefs}

\usepackage{amsmath, amssymb, amsthm}
\usepackage{esint}
\usepackage{verbatim}
\usepackage{dsfont}
\usepackage{enumitem}
\usepackage{helvet}

\DeclareMathOperator{\diam}{diam}

\DeclareMathOperator{\dist}{dist}

\newcommand{\Lm}{\mathcal{L}} 
\newcommand{\Hm}{\mathcal{H}} 
\newcommand{\N}{\mathbb{N}} 
\newcommand{\R}{\mathbb{R}} 
\newcommand{\1}{\mathds{1}} 

\newcommand{\eps}{\varepsilon}
\newcommand{\loc}{\mathrm{loc}}

\newcommand{\Per}{\mathrm{Per}}
\newcommand{\GSBV}{\mathrm{GSBV}}
\newcommand{\GSBD}{\mathrm{GSBD}}
\newcommand{\SBV}{\mathrm{SBV}}
\newcommand{\SBD}{\mathrm{SBD}}
\newcommand{\BV}{\mathrm{BV}}

\newtheorem{prop}{Proposition}

\newtheorem{lemma}{Lemma}
\newtheorem{theorem}{Theorem}
\newtheorem{rem}{Remark}

\title[Anisotrpoic Poincar\'e Inequality in $GSBV^p$ and anisotropic Mumford-Shah]{An anisotropic Poincar\'e Inequality in $GSBV^p$ and the limit of strongly anisotropic Mumford-Shah functionals}
\date{}

\begin{document}

\author[J. Ginster]{Janusz Ginster}
\address{Janusz Ginster\\ Institut f\"ur  Mathematik \\ Humboldt-Universit\"at zu Berlin \\ Unter den Linden 6 \\ 10099 Berlin, Germany}
\email{janusz.ginster@math.hu-berlin.de}

\author[P. Gladbach]{Peter Gladbach}
\address{Peter Gladbach\\ Institut f\"ur Angewandte Mathematik \\ Rheinische Friedrich-Wilhelms-Universit\"at Bonn \\ Endenicher Allee 60 \\ 53115 Bonn, Germany}
\email{gladbach@iam.uni-bonn.de}


\begin{abstract}
We show that functions in $\GSBV^p$ in three-dimensional space with small variation in $2$ of $3$ directions are close to a function of one variable outside an exceptional set.
Bounds on the  volume and the perimeter in these two directions of the exceptional sets are provided.
As a key tool we prove an approximation result for such functions by functions in $W^{1,p}$.
For this we present a two-dimensional countable ball construction that allows to carefully remove the jumps of the function.
As a direct application, we show $\Gamma$-convergence of an anisotropic three-dimensional Mumford-Shah model to a one-dimensional model.

\end{abstract}

\maketitle

\tableofcontents
\section{Introduction}
Poincar\'e inequalities in their various forms are a standard tool in the study of PDEs or variational problems.
In particular, in the Sobolev space $W^{1,p}(\Omega)$ of a connected open bounded Lipschitz domain it holds
\[
 \min_{ a \in \R} \int_{\Omega} \left|u - a\right|^p \, dx  \leq C \int_{\Omega} |\nabla u|^p \, dx.
\]
By density this can be extended to the space of functions of bounded variation, $\BV(\Omega)$, where the right hand side is replaced by $|Du|(\Omega)$, see \cite{AmFuPa}.
Clearly, in $\BV$ one cannot expect that the left hand side can be bounded in terms of the absolutely continuous part of $Du$, $\nabla u$, alone. 
However, in \cite{Leaci} De Giorgi et al. show that for functions $u \in \SBV^p(\Omega)$ whose jump set is not too large there exists an exceptional set $\omega \subset \Omega$ such that for $1\leq p < n$
\begin{equation}\label{eq: poincare leaci}
\left(\min_{a\in\R} \int_{\Omega \setminus \omega} |u-a|^{p^*} \, dx \right)^{1/p^*} \leq C \left(\int_{\Omega} |\nabla u|^p \, dx \right)^{1/p}
\end{equation}
and $\mathcal{L}^n(\omega) \leq C \Hm^{n-1}(J_u)^{n/(n-1)}$.
In \cite{Friedrich} a variant for functions in $\SBV$ whose jump set is not necessarily small is shown: In this setting there exists a Cacciopoli partition whose total boundary can be estimated by the length of the jump set of $u$ and on each set in the Cacciopoli partition the difference between $u$ and a constant can be estimated by $\nabla u$.

In the same spirit as \eqref{eq: poincare leaci}, in \cite{CaChSc} Carriero et al. show the more sophisticated Korn-Poincar\'e inequalities in the space $\GSBD^p$ (see also \cite{ChCoFr}) and even provide bounds on the perimeter of the exceptional set $\omega$.
The strategy of the proof is to find for $u \in \GSBD^p$ a function $w \in W^{1,p}$ which agrees with $u$ outside an exceptional set $\omega \subseteq \Omega$ and satisfies appropriate bounds. 
Then one can apply the ordinary Korn-Poincar\'e inequality to the more regular function $w$ to obtain an estimate on $u$ outside $\omega$.
In addition, the same article provides not only a bound on $\mathcal{L}^n(\omega)$ but also on $\operatorname{Per}(\omega)$ in terms of $H^{n-1}(J_u)$.
A simple modification of this argument yields also a version of \eqref{eq: poincare leaci} with an additional bound on the perimeter of the exceptional set, see section \ref{sec: SBD}. 

In this paper we are interested in functions in $GSBV^p$ whose jump set is approximately parallel to one of the three directions. 
For such functions it is possible to estimate the $L^p$-difference to a function only depending on this one variable outside an exceptional set satisfying appropriate bounds.
More precisely, we prove the following.

\begin{prop}\label{prop: poincareni}
 Let $S \subseteq \R^2$ be open, bounded, with Lipschitz boundary and $I \subseteq \R$ an open, bounded interval. Then there exists  for $\Omega = I \times S$ a number $\eta(\Omega) > 0$ such that the following holds:
    
    For all $p \in [1,\infty)$ there exists a constant $C(p,\Omega) > 0$ such that for all functions $u \in \GSBV^p(\Omega)$ with $\int_{J_u} |\nu'|\, d\Hm^2 \leq \eta(\Omega)$ (here $\nu = (\nu_1,\nu') \in S^2$ denotes the measure theoretic normal to $J_u$) there exists a set $\omega \subseteq \R^3$ and a measurable function $a: I \to \R$ satisfying
    \[
    \int_{\Omega \setminus \omega} |u(x) - a(x_1)|^p \, dx \leq C(p,\Omega) \int_{\Omega} |\nabla' u|^p \, dx,
    \]
    where $\nabla' = (\partial_2,\partial_3)$.
    Moreover,
     \begin{equation}
    \mathcal{L}^3(\omega)\leq   C(p,\Omega) \int_{J_u} |\nu'| \, d\Hm^2
    \end{equation}
    and
    \begin{equation} 
    \sup_{\varphi \in C^0_c(I\times S), |\varphi| \leq 1} \int_{\omega} \partial_i \varphi \, dx \leq C(p,\Omega) \int_{J_u} |\nu'| \, d\Hm^2\quad\text{ for }i=2,3.
    \end{equation}
\end{prop}

At first glance, it appears as if an application of inequality \eqref{eq: poincare leaci} in each slice $\{x_1\} \times S$ can be used to show the above result (without the bounds on the exceptional set).
The exceptional set $\omega$ would then be given by the union of the corresponding exceptional sets in each slice.
However, note that in order to guarantee $\mathcal{L}^3$-measurability of a set $\omega \subseteq \R^3$ it is not enough to ensure $\mathcal{L}^2$-measurability of every slice $(\{x_1\} \times S) \cap \omega$.
Hence, a simple application of inequality \eqref{eq: poincare leaci} or similar results in every slice cannot guarantee the measurability of the exceptional set $\omega$.

Instead we present a simple two-dimensional argument (see Theorem \ref{theorem: Sobolev approximation}) that allows us to remove the jumps of $u$ in a small neighborhood around the jump set. 
This two-dimensional construction is heavily inspired by the ball construction which was developed in the context of the Ginzburg-Landau functional, see \cite{Sandier,SandierSerfaty,Jerrard} (c.f.~\cite{DeLucaGarroniPonsiglione,AlicandroCicaleseDeLuca,Ginster,Gladbach} for an application in the context of dislocations).
We cover the jump set of $u$ by countably many balls and grow them according to the well-known ball construction until $u$ has on the outside of the grown balls boundary conditions that can be extended without jumps into the balls. 
This yields again an approximation outside a small set $\omega\subset \Omega$ from which the inequality \eqref{eq: poincare leaci} can be easily derived.
Moreover, this technique can be transferred to three dimensions by applying it simultaneously in each slice while at the same time guaranteeing measurability of the exceptional set. 


In three dimensions we prove the following approximation result.

\begin{theorem}\label{theorem: Sobolev approximation 3d}
    Let $S \subseteq \R^2$ be open, bounded, with Lipschitz boundary and $I \subseteq \R$ an open, bounded interval. Then there exists  for $\Omega = I \times S$ a number $\eta(\Omega) > 0$ such that the following hold:
    
    For all $p \in [1,\infty)$ there exists a constant $C(p,\Omega) > 0$ such that for all functions $u \in \GSBV^p(\Omega)$ with $\int_{J_u} |\nu'| \, d\Hm^2 \leq \eta(\Omega)$ (as before $\nu=(\nu_1,\nu') \in S^2$ is the measure theoretic normal to $J_u$) there exists a set $\omega \subseteq \R^3$ with 
    \begin{equation} \label{eq: est volume exceptional}
    \mathcal{L}^3(\omega)\leq   C(p,\Omega) \int_{J_u} |\nu'| \, d\Hm^2
    \end{equation}
    and
    \begin{equation} \label{eq: est surface exceptional}
    \sup_{\varphi \in C^0_c(I\times S), |\varphi| \leq 1} \int_{\omega} \partial_i \varphi \, dx \leq C(p,\Omega) \int_{J_u} |\nu'| \, d\Hm^2\quad\text{ for }i=1,2,
    \end{equation}
    and a function $w \in L^p(\Omega)$ with $w = u$ on $\Omega \setminus \omega$ and
    \begin{equation} \label{eq: est nabla w} 
    \int_{\Omega} |\nabla' w|^p \, dx \leq C(p,S) \int_{\Omega} |\nabla' u|^p \, dx,
    \end{equation}
    where $\nabla' = (\partial_2,\partial_3)$.
    \end{theorem}

Using the usual Poincar\'e inequality for $w$ on each slice for the function $a(x_1) = \fint_S w(x_1,x') dx'$ this implies Proposition \ref{prop: poincareni}.

Another consequence of Theorem \ref{theorem: Sobolev approximation 3d} is the following two-dimensional statement (which in turn implies inequality \ref{eq: poincare leaci}), which follows directly by considering functions $u(x_1,x') = u(x')$.

\begin{theorem}\label{thm: Sobolev approximation precise}
Let $S\subseteq \R^2$ be open, bounded, with Lipschitz boundary. Then there exists a number $\eta(S)>0$ and a constant $C(S)$ such that the following holds:

For all $p\in[1,\infty)$ there exists a constant $C(p,S)>0$ such that for all functions $u\in \GSBV^p(S)$ with $\Hm^1(J_u) \leq \eta(S)$ there exists a bounded set $\omega \subseteq \R^2$ with $\Per(\omega) \leq C(S) \Hm^1(J_u)$ and a function $w\in W^{1,p}(S)$ with $w=u$ in $S\setminus \omega$ and
\[
\int_S |\nabla w|^p\,dx \leq C(p,S) \int_S |\nabla u|^p\,dx.
\]
\end{theorem}

We note that the conditions $\Hm^1(J_u) \leq \eta(S)$ here and $\int_{J_u} |\nu'| \, d\Hm^2 \leq \eta(\Omega)$ in Theorem \ref{theorem: Sobolev approximation 3d} are necessary, as can be seen in Figure \ref{fig: counterexample}.

\begin{figure}
    \includegraphics{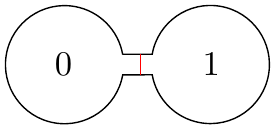}
\caption{Functions with jump set larger than $\eta(S)$ cannot be approximated with $W^{1,p}(S)$ functions in most of $\Omega$. The above function is piecewise constant but not close to any constant function.}\label{fig: counterexample}
\end{figure}

As an immediate application of Proposition \ref{prop: poincareni}, we investigate an anisotropic version of the Mumford-Shah functional (see e.g. \cite{mumford1989optimal,AmFuPa}) in a three dimensional domain $I\times S$. Here, one may consider $x_1\in I$ as a time variable and $x' = (x_2,x_3)$ as a space variable in a grayscale video file $u\in \GSBV^p(I\times S)$. We study in Section \ref{sec: Mumford Shah} the functional
\[
    F_\eps(u) := \int_{I\times S} \frac1\eps |\nabla' u|^p + |\partial_1 u |^p \,dx + \int_{J_u} \frac1\eps |\nu'| + |\nu_1|\,d\Hm^2.
    \]

    Here $\nu \in S^2$ is the measure theoretic normal to the jump set $J_u$ and $\eps >0$ is a small parameter, so that variations in $x'$ are penalized much more harshly than variations in $x_1$. As $\eps$ tends to zero, we expect compactness of finite-energy sequences, with limit functions depending only on $x_1$ and $\Gamma$-convergence of $F_\eps$ to the one-dimensional Mumford-Shah functional $F_0:L^1(I\times S)\to [0,\infty]$,
\[
    F_0(u) = \begin{cases} \Lm^2(S) \left(\int_I |u'|^p\,dx_1 + \#J_u\right) &\text{ if }u(x) = u(x_1)\in \SBV^p(I)\\
        \infty & \text{ otherwise.}
    \end{cases}    
\]

We do indeed prove this in Theorem \ref{theorem: Gamma convergence}. While $\Gamma$-convergence is obvious, the compactness statement is not. In fact, the main difficulty lies in showing that a function with finite $F_\eps$ is close in measure to a function only of $x_1$. Here we use Proposition \ref{prop: poincareni}.

The rest of the article is structured as follows: First, we introduce the necessary notation. Then we present the two-dimensional countable ball construction. For better readability we first use this construction to prove the two-dimensional result Theorem \ref{thm: Sobolev approximation precise} in Section \ref{sec: twod}. In Section \ref{sec: threed} we present the transfer of this construction to three dimensions and show Theorem \ref{theorem: Sobolev approximation 3d}. Then we showcase how this can be used to prove a $\Gamma$-convergence result for the anisotropic Mumford-Shah functional in Section \ref{sec: Mumford Shah}. 

\subsection{Notation}

We will write $C$ or $c$ for generic constants that may change from line to line but do not depend on the problem parameters.
For set of skew-symmetric matrices in $\R^{n\times n}$ we write $Skew(n)$.
For the ease of notation, we always identify vectors in $\R^n$ with their transposes.
For vectors $x = (x_1,x_2,x_3) \in \R^3$ we write $x' = (x_2,x_3)$.
Moreover, we write $\nabla' = (\partial_2,\partial_3)$.
For a measurable set $B\subseteq \R^n$ we use the notation $\mid B\mid $ or $\mathcal{L}^n(B)$ to denote its $n$-dimensional Lebesgue measure.
Similarly, we denote by $\mathcal{H}^k(B)$ its $k$-dimensional Hausdorff measure.
Moreover, we use standard notation for $L^p$-spaces and Sobolev spaces $W^{1,p}$.
In addition, we denote by $BV$ the space of functions with bounded variation, by $SBV^p$ the space of special functions of bounded variation (with $p$-integrability of the absolutely continuous part), and by $GSBV^p$ the space of generalized special functions of bounded variation as introduced in \cite{AmFuPa}.
In particular, we use for a function $u \in SBV^p$ the decomposition
\[
Du = \nabla u \, d\mathcal{L}^n + [u] \otimes \nu \, d\mathcal{H}^{n-1}_{|J_u},
\]
where $J_u$ is the jump set of $u$, $\nu \in S^{n-1}$ is the measure theoretic normal to $J_u$ and $[u] = u^+ - u^-$ for the approximate limits $u^+$ and $u^-$ on $J_u$.
In the specific case of an indicator function $u = \1_{\omega}$ for a measurable $\omega \subseteq \R^n$ we write $\Per(\omega) := |Du|(\R^n)$ for its perimeter.
Eventually, we write $GSBD^p$ for the space of specialized special functions of bounded deformation as introduced in \cite{dal2013generalised}.
In particular, for $u \in  GSBD^p$ we denote by $e(u)$ the density of the absolutely continuous part of the symmetrized derivative $Eu$.
In the specific case that $u \in W^{1,p}$, this means $e(u) = \frac{\nabla u + \nabla u ^T}2$.

\section{Two-dimensional result: Proof of Theorem \ref{thm: Sobolev approximation precise}} \label{sec: twod}

\subsection{A Proof using Korn's Inequality in $\SBD^p$}\label{sec: SBD}

One way to prove Theorem \ref{thm: Sobolev approximation precise} is to use the existing (and more complicated) approximation results in $\SBD^p$ which were used to to prove a Korn inequality in $\SBD^p$.

\begin{proof}[Proof of Theorem \ref{thm: Sobolev approximation precise}]
Let $u \in \GSBV^p(S)$. We then define the vector-valued function $v \in  \GSBD^p(S)$ by $v = (u,0)$. By \cite[Theorem 1.2]{CaChSc} there exists $\omega \subseteq S$ with $\mathcal{L}^2(\omega)^{\frac12} + \Hm^1(\partial^* \omega) \leq C \Hm^1(J_v) = C \Hm^1(J_u)$ and a function $w \in W^{1,p}(S;\R^2)$ such that $v = w$ on $S\setminus \omega$ and
\begin{align*}
\int_S |e(w)|^p \, dx \leq C \int_S |e(v)|^p \, dx.
\end{align*} 
By Korn's inequality applied to $w$ there exists $W \in Skew(2)$ such that
\begin{align*}
\int_S |\nabla w - W|^p \, dx \leq C \int_S |e(v)|^p \, dx.
\end{align*}
In particular, we find
\begin{align}
\mathcal{L}^2(\mathcal{S \setminus \omega}) |W_{21}|^p \leq \int_{S \setminus \omega} \left| \begin{pmatrix} \partial_1 u && \partial_2 u \\ 0 && 0
\end{pmatrix} - W \right|^p \, dx \leq C \int_S |e(v)|^p.
\end{align}
If $\mathcal{L}^2(\omega) \leq \frac12 \mathcal{L}^2(S)$, i.e. $\Hm^1(J_u)$ is small enough, we obtain $|W_{21}|^p \leq \frac{2C}{\mathcal{L}^2(S)} \int_S |e(v)|^p$.
As $W$ is skew-symmetric this already implies $|W|^p \leq \frac{2C}{\mathcal{L}^2(S)} \int_S |e(v)|^p$.
Hence, by the triangle inequality we find
\begin{align}
\int_{S} |\nabla w|^2 \, dx &\leq 2^p \int_S |\nabla w - W|^p \, dx + 2^p \int_S |W|^p \, dx \\
& \leq C \int_S |e(v)|^p \, dx \\ 
&\leq  C \int_S |\nabla u|^p \, dx,
\end{align}
which concludes the proof.
\end{proof}

We close this section with the following remark.

\begin{rem}
Later we would like to apply this result in each slice of the form $\{x_1\} \times S$ to obtain exceptional sets $\omega_{x_1} \subseteq S$ and functions $w_{x_1} \in W^{1,p}(S)$ satisfying the conclusion of Theorem \ref{thm: Sobolev approximation precise}.
However, it is not immediate from this proof that one can define a \emph{measurable} set $\omega = \bigcup_{x_1} \{x_1\} \times S$ and a function \emph{measurable} $w: I \times S \to \R$ through $w(x_1,x_2,x_3) = w_{x_1}(x_2,x_3)$.
As the construction in \cite{CaChSc} is rather involved, we present a simpler method in this paper to construct the desired set and function simultaneously in all slices using the ball construction technique.
From our construction it is easy to see that all defined quantities are measurable.
\end{rem}

\subsection{The countable ball construction and the proof of Theorem \ref{theorem: Sobolev approximation}}

In this section, we prove the following more general version of Theorem \ref{thm: Sobolev approximation precise}:

\begin{theorem}\label{theorem: Sobolev approximation}
    Let $S\subseteq \R^2$ be open, bounded, with Lipschitz boundary. Then there exists a number $ \eta(S) > 0$ and a constant $C(S)$  such that the following holds:
    
    For all $p\in [1,\infty)$ there exists a constant $C(p,S)$ such that for all $T>0$ and all functions $u\in \GSBV^p(S)$ with $\Hm^1(J_u) \leq e^{-T}\eta(S)$ there exists a bounded set $\omega \subseteq \R^2$ with $\Per(\omega) \leq e^T C(S) \Hm^1(J_u)$ and a function $w\in W^{1,p}(S)$ with $w=u$ in $S\setminus \omega$ and
    \[
    \int_S |\nabla w|^p\,dx \leq \left( 1+ \frac{C(p,S)}{T}\right) \int_S |\nabla u|^p\,dx.
    \]
    \end{theorem}

Theorem \ref{thm: Sobolev approximation precise} follows from Theorem \ref{theorem: Sobolev approximation} by taking $T=1$.

Before we prove Theorem \ref{theorem: Sobolev approximation}, we briefly introduce the ball-construction technique from the seminal paper \cite{Sandier}:

\begin{lemma}\label{lemma: finite balls}
Given a finite set of balls $B_i := B(x_i,r_i) \subset \R^d$, $i=1,\ldots,N$, there is a family of finite sets of balls $B_i^t := B(x_i^t,r_i^t)$, $i=1,\ldots,N$, $t\in[0,t_i)$ with collapse times $t_i\in [0,\infty]$ such that for $I_t:= \{i=1,\ldots,N\,;\, t < t_i \}$ we have
\begin{itemize}
\item [(i)] $\bigcup_{i=1}^N B_i \subseteq \bigcup_{i\in I_s} B_i^s \subset \bigcup_{i\in I_t}B_i^t$ whenever $0\leq s \leq t < \infty$.
\item [(ii)] $\sum_{i\in I_t} r_i^t = e^t \sum_{i=1}^N r_i$.
\item [(iii)] $(\overline{B_i^t})_{i\in I_t}$ is pairwise disjoint.
\item [(iv)] $B(x_i^s,e^{t-s}r_i^s) \subseteq B_i^t$ whenever $0\leq s \leq t < t_i$.
\end{itemize}
\end{lemma}

\begin{proof}
We first replace the balls $B_i$ with balls $B_i^0$ such that (iii) holds. If $\overline{B_i} \cap \overline{B_j}\neq \emptyset$, $i<j$, then $|x_i - x_j| \leq r_i + r_j$. But then
\[
B_i \cup B_j \subset B(x_i^0,r_i^0), \text{ where }x_i^0 := \frac{r_i}{r_i+r_j}x_i + \frac{r_j}{r_i + r_j} x_j\text{ and }r_i^0 := r_i + r_j.
\]
Also we set $t_j := 0$.

We may repeat this replacement at most $N-1$ times until (iii) is satisfied, defining $I_0$ and $(B_i^0)_{i\in I_0}$. We define $I_t := I_0$ and $B_i^t := B(x_i^0, e^t r_i^0)$ for $t<\eps$, where
\[
\eps := \inf\{t>o\,:\,\overline{B_i^t} \cap \overline{B_j^t}\neq \emptyset\text{ for some }i\neq j\in I_0 \}.
\]

If $\eps = \infty$, we are done. Otherwise, at time $\eps$, we repeat first the replacement and then the growing scheme to some larger time. Since at every such time the number of balls is strictly decreased, there are at most $N-1$ such collapsing times, at which point a single ball remains.

This defines $I_t$ and the balls $(B_i^t)_{i\in I_t}$ for all $t\in[0,\infty)$.

It is straightforward to check that (i)-(iv) are satisfied for all times.

\end{proof}

\begin{rem}
We see that as $N$ is replaced by $N+1$, all collapsing times $t_i$ decrease and all radii increase. This is important in the following:
\end{rem}

We wish to apply a version of Lemma \ref{lemma: finite balls} to a covering of the jump set $J_u$. However, as the jump set of a function $u\in \GSBV^p(S)$ is in general not compact, countably many balls are required to cover it. We can however extend the ball construction to the case of countably many balls with $\sum_{i\in \N} r_i < \infty$, c.f.~also \cite{Gladbach}:

\begin{lemma}\label{lemma: balls}
Given an at most countable set of balls $B_i := B(x_i,r_i) \subset \R^d$, $i\in \N$ with
\[
\sum_{i\in \N} r_i < \infty,
\]
there is a family of at most countable sets of balls $B_i^t := B(x_i^t,r_i^t)$, $i\in \N$, $t\in[0,t_i)$ with collapse times $t_i\in [0,\infty]$ such that for $I_t:= \{i\in \N \,;\, t < t_i \}$ we have
\begin{itemize}
\item [(i)] $\bigcup_{i \leq N} B_i \subseteq \bigcup_{i \leq N ,i\in I_s} B_i^s \subseteq \bigcup_{i \leq N ,i\in I_t} B_i^t$ whenever $0\leq s \leq t < \infty$, for all $N\in \N$.
\item [(ii)] $\sum_{i\in I_t} r_i^t \leq e^t \sum_{i\in \N} r_i$ for all $t>0$.
\item [(iii)] For all $t>0$ the family $(\overline{B_i^t})_{i\in I_t}$ is pairwise disjoint.
\item [(iv)] $B(x_i^s,e^{t-s}r_i^s) \subseteq B_i^t$ whenever $0\leq s \leq t < t_i$.
\item [(v)] The balls $B_i^t$ are monotone in the following sense: If $B_i \subseteq B_i'$ for every $i\in \N$, then $t_i' \leq t_i$ and $B_i^t \subseteq (B_i^t)'$ for every $i\in \N$, $t < t_i'$. Also
\[
\bigcup_{i \leq N, i \in I_t} B_i^t \subseteq \bigcup_{i\leq N, i \in I_t'} (B_i^t)'\text{ for all }N\in \N, t \geq 0.
\]
\end{itemize}
\end{lemma}

Note that the first condition implies in particular that $t\mapsto \bigcup_{i\in I_t} B_i^t$ is monotone. In addition, the index $i(t)$ of the ball covering $B_i^s$ decreases with $t$.

\begin{proof}
We start with the situation where $r_i = 0$ for all $i>N$ for some $N\in\N$. We show (i)-(v) in this case by induction over $N$. Consider first the case $N=0$, where we simply set $t_0 := \infty, x_0^t = x_0, r_0^t := e^t r_0$ for all $t\geq 0$. In particular we note that if $B_0 \subseteq \tilde B_0$, then $B_0^t \subseteq \tilde B_0^t$.

Now assume (i)-(v) hold for $N-1$. Let $\tilde t_i\in [0,\infty]$, $\tilde B_i^t \subset \R^d$ for $i=0,\ldots,N-1$, $t\in[0,\tilde t_i)$, be a family of sets of balls for $B_1,\ldots,B_{N-1}$. Consider $B_N^t:= B(x_N, e^t r_N)$ for $t\geq 0$. Let 
\[
t_N:= \inf\{t\geq 0 \,:,\, \overline B_N^t \cap \bigcup_{i=0}^{N-1} \overline{\tilde B_i^t} \neq \emptyset \}.
\]

By our construction, there is $i_0 \in \{0,\ldots,N-1\}$ such that $\tilde t_i > t_N$ and $\overline{B_N^{t_N}} \cap \overline{\tilde B_i^{t_N}} \neq \emptyset$.

We now set $B_i^t := \tilde B_i^t$ for $i=0,\ldots,N-1$, $t< \min(\tilde t_i, t_N)$.

At time $t_N$, we replace any two balls $B = B(x,r)$, $B' =B(x',r')$ among $\{B_i^{t_N}\,:\,i=0,\ldots,N_1, t_i > t_n\} \cup \{B_N^{t_N}\}$ with $\overline{B} \cap \overline{B'}\neq \emptyset$ with the new ball $B'':= B(x'',r'')$, where
\[
x'':= \frac{r}{r+r'}x + \frac{r'}{r+r'} x' \quad, \quad r'' := r + r'.
\]
We note that if $\overline{B} \cap \overline{B'}\neq \emptyset$, then $|x-x'| \leq r+r'$ and thus $B \cup B' \subseteq B''$. Also, if $B\subseteq \tilde B$ and $B' \subseteq \tilde B'$ then $B'' \subseteq \tilde B''$.

We repeat the replacement of two balls as above until $\overline{B} \cap \overline{B'} = \emptyset$ for all balls in the set. As the replacement happens at least once, namely $\overline{\tilde B_{i_0}^{t_N}} \cap \overline{B_N^{t_N}} \neq \emptyset$, there are at most $N-1$ balls remaining after replacements. The replacements are monotone in the sense of (v). From time $t_N$ onwards, we restart the growing of balls and see that (i)-(v) remain satisfied.

This finishes the proof in the case of finitely many balls.

For the case of countably many balls, we perform the finite construction above for every $N\in \N$, yielding nonincreasing sequences of collapse times $t_{i,N}\in [0,\infty], N\geq i$ and nondecreasing sequences of families of balls $B_{i,N}^t \subseteq \R^d, N\geq i, t< t_{i,N}$. Because of the monotonicities there exist limits for the collapse times $t_i := \lim_{N\to \infty} t_{i,N}$ and limit balls $\tilde B_i^t := \bigcup_{N\to \infty} B_{i,N}^t$ for $t<t_i$. This is indeed a ball because it is the countable union of a nondecreasing sequence of balls, where $r_{i,N}^t \leq e^t \sum_{j\in \N} r_j$ is bounded, so that the union is in fact an open ball.

The family thus created unfortunately does not satisfy (i), (iii), or (v). In order to make them true, we define
\[
B_i^t:= \left(\bigcap_{t<r<t_i} \tilde B_i^r\right)^\circ\text{ for }t<t_i.
\]
We now go through the properties (i)-(v).

For (i), we fix $N\in \N$. Then
\[
\bigcup_{i\leq N,i\in I_t} B_i^t = \bigcup_{i\leq N, i \in I_t} \left(\bigcap_{t<r<t_N} \tilde B_i^r \right)^\circ = \left( \bigcap_{t<r<\min\{t_i\,:\,i\leq N, t_i > t \}} \bigcup_{i\leq N, r<t_i} \tilde B_i^r   \right)^\circ.
\]
This shows (i).

For (ii) by Fatou's Lemma
\[
\sum_{i\in I_t} r_i^t \leq \liminf_{r \searrow t} \sum_{i\in I_r} \tilde r_i^r \leq \liminf_{r\searrow t} \limsup_{N\to \infty} \sum_{i\in I_{r,N}} r_{i,N}^r \leq  e^t \sum_{i\in \N} r_i.
\]

For (iv), since (iv) holds for $N\in\N$, we have $\tilde B(x_i^{s+\eps},e^{t-s}\tilde r_i^{s+\eps}) \subseteq \tilde B_i^{t+\eps}$ for every $\eps>0$. Taking the limit $\eps\to 0$, it follows that $B(x_i^s, e^{t-s}r_i^s) \subseteq B_i^t$.

For (iii), take two balls $B_i^t, B_j^t$ with $i\neq j\in I_t$ and assume that $\overline{B_i^t} \cap \overline{B_j^t} \neq \emptyset$. Since (iv) holds, we also have $\tilde B_i^{t+\eps}\cap \tilde B_j^{t+\eps} \neq \emptyset$ for $\eps>0$ small enough. But then already $B_{i,N}^{t+\eps} \cap B_{j,N}^{t,\eps} \neq \emptyset$ for $N$ large, a contradiction.

Finally, (v) is clear from the construction.

\end{proof}

With this, we can prove the following:

\begin{lemma}\label{lemma: balls Fubini}
Given a family of balls  $B_i^t = B(x_i^t,r_i^t)$ satisfying (i),(iii),(iv) from Lemma \ref{lemma: balls} and a nonnegative function $f\in L^1(\R^d)$, we have
\begin{equation}
\int_0^\infty \sum_{i\in I_t} r_i^t \int_{\partial B_i^t} f\,d\Hm^{d-1} \,dt \leq \int_{\R^d} f(x)\,dx.
\end{equation}

In particular, the surface integral is well-defined for almost every $t\in [0,\infty)$.
\end{lemma}

\begin{proof}
We define for $N\in \N$ the function $g_N:[0,\infty) \to [0,\infty)$,
\[
g_N(t) := \sum_{i\in I_t, i\leq N} \int_{B_i^t} f(x)\,dx.
\]

Then by (i) every $g_N$ is nondecreasing and therefore differentiable almost everywhere. Then we have by (iii)
\[
\int_{\R^d} f(x)\,dx \geq \lim_{T\to \infty} g_N(T) \geq \int_0^\infty g_N'(t)\,dt.
\]

If $g_N'(t)$ exists and $t\neq t_i$ for any $i\leq N$ , then by (iv)
\begin{align*}
g_N'(t) = \lim_{h\searrow 0} \sum_{i\in I_t,i\leq N} \frac1h \int_{B_i^{t+h} \setminus B_i^t} f(x)\,dx &\geq \lim_{h\searrow 0} \sum_{i\in I_t,i\leq N} \frac1h \int_{B(x_i^t, e^hr_i^t) \setminus B_i^t} f(x)\,dx \\ &= \sum_{i\in I_t,i\leq N} r_i^t \int_{\partial B_i^t} f\,d\Hm^{d-1}.
\end{align*}

By Fatou's Lemma we obtain
\[
\int_0^\infty \sum_{i\in I_t} r_i^t \int_{\partial B_i^t} f\,d\Hm^{d-1} \leq \int_{\R^d} f(x)\,dx,
\]
where for every $i\in \N$ we have
\[
\int_{\partial B_i^t} f\,d\Hm^{d-1} = \lim_{h \searrow 0} \frac1h \int_{B(x_i^t,r_i^t + h) \setminus B_i^t} f(x)\,dx
\]
for almost every $t\in[0,t_i)$, which completes the proof.
\end{proof}

We can now finally prove Theorem \ref{theorem: Sobolev approximation}:

\begin{proof}[Proof of Theorem \ref{theorem: Sobolev approximation}]
We begin by noting that there exists an open bounded set $S' \subset \R^2$ with Lipschitz boundary such that $S\Subset S'$ and for any $u\in \GSBV^p(S)$ there exists an extension $U\in \GSBV^p(S')$ with $U = u$ in $S$ and
\[
\int_{S'} |\nabla U|^p\,dx \leq C(S) \int_S |\nabla u|^p\,dx,\quad \Hm^1(J_{U}) \leq C(S) \Hm^1(J_u).
\]

We wish to cover the singular set
\[
S_U := \{x\in S'\,:\,U\text{ is not approximately continuous at }x\},
\]
where we recall that $U$ is approximately continuous at $x$ if for any $\eps>0$, we have
\[
\lim_{r\searrow 0} \frac{|\{y\in B(x,r)\,:\,|u(y)-u(x)|>\eps\}|}{r^2} = 0.
\]

Recall that $\Hm^1(S_U\setminus J_U) = 0$ and thus $\Hm^1(S_U) = \Hm^1(J_U)\leq C(S) \Hm^1(J_u)$. By Vitali's covering theorem, we can cover $S_U$ with countably many balls $B_i = B(x_i,r_i) ,i\in\N$ such that $\sum_{i\in\N} r_i \leq C(S)\Hm^1(J_u)$. We apply Lemma \ref{lemma: balls} to this family, yielding $B_i^t$ satisfying (i)-(v). We then apply Lemma \ref{lemma: balls Fubini} to this family with $f = |\nabla U|^p \mathds{1}_{S'}\in L^1(\R^2)$. We then have
\begin{equation}\label{eq: Fubini}
\int_0^\infty \sum_{i\in I_t} r_i^t \int_{\partial B_i^t \cap S'} |\nabla U|^p\,d\Hm^1 \,dt\leq C(S) \int_S |\nabla u|^p\,dx.
\end{equation}

By the integral mean value theorem applied to \eqref{eq: Fubini}, we find $t_0\in (0,T)$ such that
\begin{equation}\label{eq: t_0}
 \sum_{i\in I_{t_0}} r_i^{t_0} \int_{\partial B_i^{t_0} \cap S'} |\nabla U|^p\,d\Hm^1 \leq \frac{1}{T} \int_S |\nabla u|^p\,dx.
\end{equation}

By (ii) we have
\begin{equation}\label{eq: radii}
\sum_{i\in I_{t_0}} r_i^{t_0} \leq C(S) e^T \Hm^1(J_u).
\end{equation}

In particular setting $\omega:= \bigcup_{i\in I_{t_0}} B_i^{t_0}$ we find that $\Per(\omega) \leq 2\pi C(S) e^T \Hm^1(J_u)$.

We further define $I := \{i\in I_{t_0}\,:\,B_i^{t_0} \cap S \neq \emptyset\}$. Because by \eqref{eq: radii} any $\diam(B_i^{t_0}) = 2 r_i^{t_0} \leq 2 C(S) e^T \Hm^1(J_u) \leq \dist(S,\partial S')$ if $e^T\Hm^1(J_u) \leq \frac{\dist(S,\partial S')}{2 C(S)} =: \eta(S)$, for all $i\in I$ we have $\partial B_i^t \subset S'$. This allows us to define $w:=u=U$ in $S\setminus \bigcup_{i\in I} B_i^{t_0}$ and using polar coordinates to write any point $y\in B_i^{t_0}$ as $y=(1-\theta) x_i^{t_0} +\theta z$ for $z\in \partial B_i^{t_0}$ and $\theta\in [0,1)$ we define $w$ in $B_i^{t_0}$ as
\begin{equation}
w((1-\theta) x_i^{t_0} +\theta z) := (1-\theta) \fint_{\partial B_i^{t_0}} U\,d\Hm^1 + \theta U(z). 
\end{equation}

A direct calculation shows that
\begin{equation}
\int_{B_i^{t_0}} |\nabla w|^p\,dx \leq C(p) r_i^{t_0} \int_{\partial B_i^{t_0}} |\nabla U|^p\,dx,
\end{equation}
with $C(p) = 1 + \pi^{p+1}$.

Summing up over all $i\in I$ we obtain
\begin{equation}
\int_{\omega \cap S} |\nabla w|^p \leq \frac{C(p,S)}{T} \int_S |\nabla u|^p\,dx.
\end{equation}

Finally we have to show that $w\in W^{1,p}(S)$. Clearly $w$ is the pointwise almost everywhere limit of the functions
\[
w_N(x):= \begin{cases}
w(x) & x\in B_i^{t_0}, i\in I, i\leq N\\
u(x) &\text{otherwise.}
\end{cases}
\]

We see that the singular set of $w_N$ is given by $S_{w_N} = S_u \cap \bigcup_{i\in I, i>N} B_i^{t_0}$. By (i) however,
\begin{equation}
\Hm^1(S_u \cap \bigcup_{i\in I, i>N} B_i^{t_0}) \leq \Hm^1(S_u \cap \bigcup_{i>N} B_i) \to_{N\to \infty} 0.
\end{equation}

By the $\GSBV$ compactness theorem, we have that $\Hm^1(S_w) = 0$ and thus $w\in W^{1,p}(S)$. This concludes our proof.
\end{proof}

\section{Three-dimensional result: Proof of Theorem \ref{theorem: Sobolev approximation 3d}} \label{sec: threed}

Theorem \ref{theorem: Sobolev approximation 3d} follows immediately from Theorem \ref{theorem: Sobolev approximation 3d precise} below by setting $T=1$.

\begin{theorem}\label{theorem: Sobolev approximation 3d precise}
Let $S \subseteq \R^2$ be open, bounded, with Lipschitz boundary and $I \subseteq \R$ an open, bounded interval. Then there exists  for $\Omega = I \times S$ a number $\eta(\Omega) > 0$ such that the following hold:

For all $p \in [1,\infty)$ there exists a constant $C(p,\Omega) > 0$ such that for all $T >0$ and all functions $u \in \GSBV^p(\Omega)$ with $\int_{J_u} |\nu'| \, d\Hm^2 \leq e^{-T}\eta(\Omega)$ there exists a set $\omega \subseteq \R^3$ with 
\begin{equation} \label{eq: est volume exceptional precise}
\mathcal{L}^3(\omega)\leq   C(p,\Omega)e^T \int_{J_u} |\nu'| \, d\Hm^2
\end{equation}
and
\begin{equation} \label{eq: est surface exceptional precise}
\sup_{\varphi \in C^0_c(I\times S), |\varphi| \leq 1} \int_{\omega} \partial_i \varphi \, dx \leq C(p,\Omega)e^T \int_{J_u} |\nu'| \, d\Hm^2\quad{ for }i=2,3,
\end{equation}
and a function $w \in L^p(\Omega)$ with $w = u$ on $\Omega \setminus \omega$ and
\begin{equation} \label{eq: est nabla w precise} 
\int_{\Omega} |\nabla' w|^p \, dx \leq \left(1+ \frac{C(p,S)}{T}\right) \int_{\Omega} |\nabla' u|^p \, dx.
\end{equation}
\end{theorem}

\begin{rem}
We note that the estimates \eqref{eq: est volume exceptional precise} and \eqref{eq: est surface exceptional precise} are optimal in the sense that one cannot bound $\mathcal{L}^3(\omega)$ or $\sup_{\varphi \in C^0_c(I\times S), |\varphi| \leq 1} \int_{\omega} \partial_i \varphi \, dx$ by $\left(\int_{J_u} |\nu'|\, d\Hm^2\right)^{\alpha}$ for some $\alpha > 1$. 
Indeed: Let $I=(-1,1)$, $S = B(0,1) \subseteq \R^2$ and $h>0$.
\begin{enumerate}
\item[(a)]  Consider the function $u \in \SBV^p(I\times S)$ given by 
\begin{equation}
u(x) = \begin{cases} 0 &\text{ if } x \in (-h/2,h/2) \times B(0,1/2) \\ 1 &\text{ else.}  \end{cases}
\end{equation}

Then $\nabla u = 0$ and $\int_{J_u} |\nu'| \, d\Hm^2 \leq C h$.
It follows that any function $w$ satisfying \eqref{eq: est nabla w} is constant on each slice $\{x_1\} \times S$.
Hence, the exceptional set $\omega$ satisfies $(-h/2,h/2) \times B(0,1/2) \subseteq \omega$ or $(-h/2,h/2) \times (B(0,1) \setminus B(0,1/2)) \subseteq \omega$. 
In particular, $\mathcal{L}^3(\omega) \geq c h$. This shows that \eqref{eq: est volume exceptional precise} is optimal.
\item[(b)] Now, consider the function $u \in \SBV^p(I \times S)$ defined by
\[
\begin{cases} 0 &\text{ if } x \in (-1,1) \times B(0,h) \\ 1 &\text{ else.}  \end{cases}
\] 

Again, $\nabla u = 0$, $\int_{J_u} |\nu'| \, d\Hm^2 \leq Ch$ and any $w$ satisfying \eqref{eq: est nabla w} is constant on each slice. 
If the exceptional set $\omega$ satisfies \eqref{eq: est volume exceptional precise} then we have for small $h$ necessarily that $w=1$. Hence, we have on almost every slice $\{x_1\} \times B(0,h) \subseteq \omega$. 
On the other hand, on at least half of the slices we have $\mathcal{L}^2(\{x' \in S: (x_1,x') \in \omega\} ) \leq Ch$.
This implies that $ \max_{i=2,3} \sup_{\varphi \in C^0_c(I\times S), |\varphi| \leq 1} \int_{\omega} \partial_i \varphi \, dx \geq c h$, which shows that \eqref{eq: est surface exceptional precise} is optimal.

\end{enumerate}
\end{rem}

\begin{rem}
Note that the $1+1$-dimensional version of Theorem \ref{theorem: Sobolev approximation 3d} is much simpler. 
For $\Omega \subseteq \R^2$ and $u \in \SBV^p(\Omega)$ we define $\omega := \{ (x_1,x_2) \in \Omega : (\{x_1\} \times \R) \cap J_u \neq \emptyset \}$. 
Then $u$ is absolutely continuous in the $x_2$ direction on $\Omega \setminus \omega$. Moreover, $\Lm^2(\omega) \leq \int_{J_u} |\nu_2|\, d\Hm^1$ and $\sup_{\varphi \in C^0_c(\Omega), |\varphi| \leq 1} \int_{\omega} \partial_2 \varphi \, dx =0$.
\end{rem}

\begin{proof}[Proof of Theorem \ref{theorem: Sobolev approximation 3d}]
We will argue as in the proof of Theorem \ref{theorem: Sobolev approximation} simultaneously in every slice $\{x_1\} \times S$ of $\Omega = I \times S$.
Again, we find an open bounded set $S' \subset \R^2$ with Lipschitz boundary such that $S \Subset S'$ and for any $u\in \GSBV^p(\Omega)$ there exists an extension $U\in \GSBV^p(I \times S')$ with $U = u$ in $\Omega$ and
\begin{align}
&\int_{I \times S'} |\nabla U|^p\,dx \leq C(\Omega) \int_{\Omega} |\nabla u|^p\,dx, \\
&\Hm^2(J_{U}) \leq C(\Omega) \Hm^2(J_u) \\ 
\text{ and } &\int_{J_U} |\nu'| \,d\Hm^2 \leq C(S) \int_{J_u} |\nu'| \,d\Hm^2.
\end{align}

Next, set $\delta:= \left( \Hm^2(J_U) \right)^{-1} \int_{J_U} |\nu'| \, d\Hm^2$ and define $U_{\delta}: I \times \delta S \to \R$ by $U_{\delta}(x_1,x_2,x_3) = U(x_1,x_2/\delta,x_3/\delta)$.
It follows that
\begin{align}
\Hm^2(J_{U_\delta}) = \int_{J_U} \left| ( \delta^2 \nu_1, \delta\nu_2, \delta\nu_3 ) \right| \, d\Hm^2 &\leq \delta^2 \Hm^2(J_U) + \delta \int_{J_U} |\nu'| \, d\Hm^2 \\ &= 2  \delta \int_{J_U} |\nu'| \, d\Hm^2.
\end{align}
 
We use Vitali's covering theorem to cover $J_{U_\delta}$ and obtain countably  many balls $\tilde{B}_i=B(\tilde{x}_i,\tilde{r}_i)$ such that $\sum_{i \in \N} \tilde{r}_i^2 \leq C \Hm^2(J_{U_\delta})$.
Let us then define the cylinders 
\[
\tilde{C}_i = ((\tilde{x}_i)_1-\tilde{r}_i,(\tilde{x}_i)_1+\tilde{r}_i) \times \pi\left( \tilde{B}_i \right),
\]
where $\pi: \R^3 \to \R^2$ is the projection along the $x_1$-variable.
Then we define the scaled cylinders 
\[
C_i = \{ x \in I \times S': (x_1,\delta x_2,\delta x_3) \in \tilde{C}_i \}.
\]

We denote for $x_1 \in I$ by 
\[
r_{i,x_1} = \frac{\tilde{r}_i}{\delta} \, \1_{\{s \in I: (\{s\} \times S) \cap C_i \neq \emptyset\}}(x_1)
\]
the radius corresponding to the cylinder $C_i$ if the slice $(\{x_1\} \times S')$ intersects $C_i$.
Then we estimate 
\begin{align}
2C\delta \int_{J_U}|\nu'| \, d\Hm^2 \geq 2\sum_{i} \tilde{r}_i^2 &\geq \int_I \sum_i \tilde{r}_i \, \1_{\{s \in I: (\{s\} \times S) \cap C_i \neq \emptyset\}}(x_1) \, dx_1 \\ &\geq 
\delta \, \int_I \sum_i r_{i,x_1} \, dx_1. \label{eq: estimate sum r_i}
\end{align}

Let us then write $h := \int_{J_U}|\nu'| \, d\Hm^2$.
Moreover, define the sets $I_{ad} = \{ x_1 \in I: \sum_i r_{i,x_1}  \leq \eta(\Omega)\}$ and $I_{ad}^N = \{ x_1 \in I: \sum_{i=1}^N r_{i,x_1}  \leq \eta(\Omega)  \}$, where $\eta(\Omega)>0$ will be fixed later.
It follows immediately that $|I_{ad}^N| \geq |I_{ad}| \geq |I| - C \frac{h}{\eta(\Omega)}$.

Now, we apply simultaneously for all $x_1 \in I$ Lemma \ref{lemma: balls} to the balls $B_{i,x_1} := \pi\left((\{x_1\} \times S') \cap C_i\right)$ in $S'$ to obtain a family of at most countably many balls $B_{i,x_1}^t$ with radii $r_{i,x_1}^t$ and collapse times $t_{i,x_1}$ satisfying (i) to (v) of Lemma \ref{lemma: balls}.
Let us then define for $t>0$ the set 
\[
E_t = \bigcup_{x_1 \in I_{ad}} \bigcup_i \{x_1\} \times \overline{B_{i,x_1}^t}.
\]

We claim that the set $E_t$ is measurable for all $t$. 
As in the proof of Lemma \ref{lemma: balls} we first start with the first $N$ cylinders and perform the construction from Lemma \ref{lemma: finite balls} in each slice $\{x_1\} \times S'$.
For an illustration see Figure \ref{fig: construction cylinders}.
We denote by $B_{i,x_1}^{t,N}\subseteq \R^2$ the corresponding balls with radius $r_{i,x_1}^{t,N}$ and centers $x_{i,x_1}^{t,N}$.
Moreover, we write $E_t^N \subseteq I \times S'$ for the corresponding set at time $t$.
Note that the set $E^N_t$ is by the very construction the union of finitely many cylinders (potentially more than at the beginning, see Figure \ref{fig: construction cylinders}).
We denote these cylinders at time $t$ by $C_{i}^{t,N}$ with radius $r_{i}^{t,N}$.

As in the proof of Lemma \ref{lemma: balls} the pointwise limit $\lim_{N\to\infty} \1_{E^N_t}$ exists by monotonicity for every $t\geq 0$ and is again an indicator function.
We call the corresponding measurable set $\tilde{E}_t \subseteq I \times S'$.
Again, as in the proof of Lemma \ref{lemma: balls} it holds in a pointwise almost everywhere sense $\lim_{s \searrow t} \1_{\tilde{E}_s} \1_{I_{ad} \times S'} = \1_{E_t}$.
In particular, $E_t$ is measurable.

\begin{figure}
\includegraphics[scale=0.7]{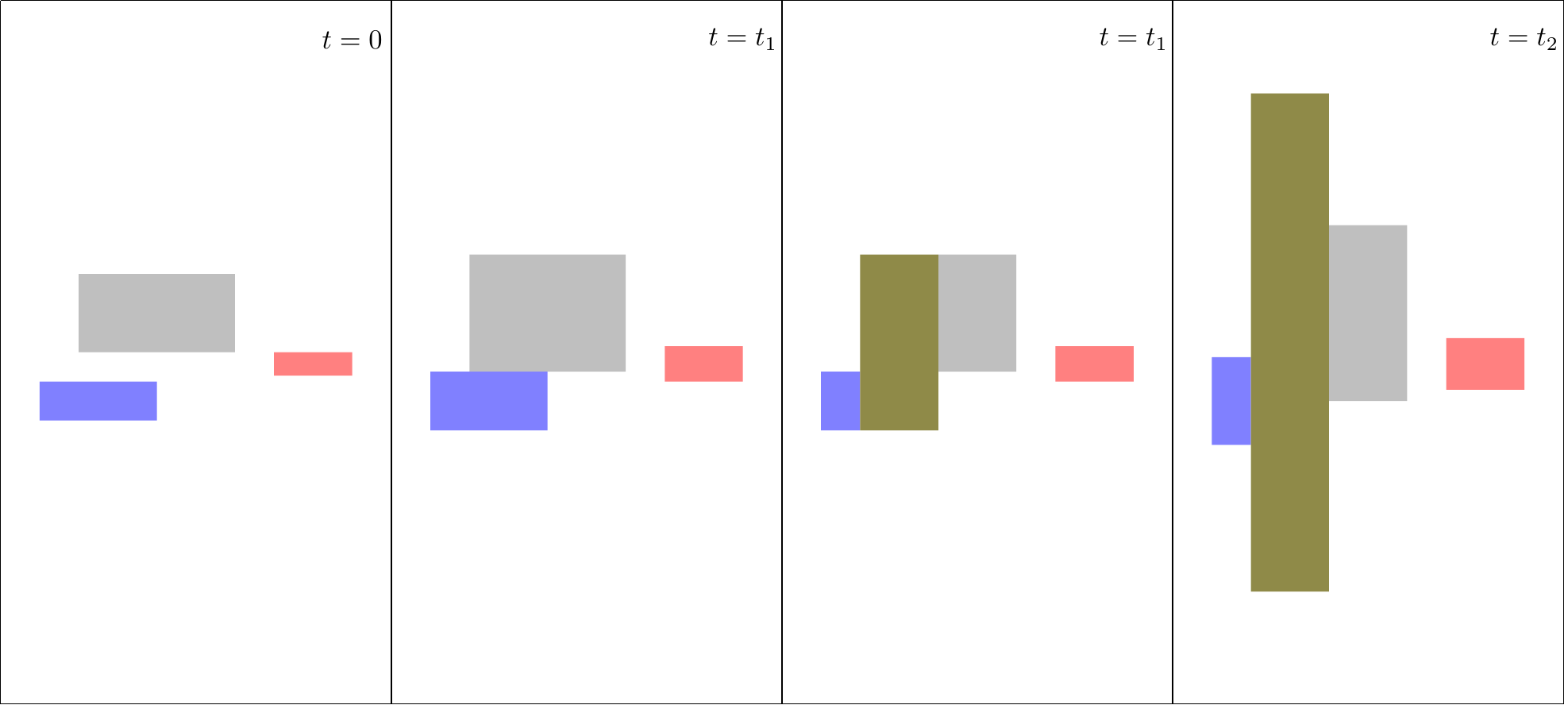}
\caption{A sketch of the construction in the $x_1$-$x_2$-plane, i.e. the cylinders appear to be rectangles. At time $t=0$ we start with finitely many cylinders (grey, blue, red). Until time $t_1$ in every slice $\{x_1\}\times S'$ all appearing balls expand, i.e., the corresponding cylinders grow in the $x_2$-$x_3$-direction and stay cylinders. At time $t_1$ in certain slices $\{x_1\} \times S'$ balls are merged. The corresponding set $E^N_{t_2}$ can again be decomposed into finitely many cylinders (grey, blue, red, yellow). Then these new cylinders continue to grow in the $x_2$-$x_3$-direction.  }
\label{fig: construction cylinders}
\end{figure}

Additionally, we estimate using Fubini, the isoperimetric inequality in dimension two, the definition of $I_{ad}$, estimate (ii) of Lemma \ref{lemma: balls} and estimate \eqref{eq: estimate sum r_i}
\begin{align}
\mathcal{L}^3(E_t) &= \int_I \mathcal{L}^2((E_t)_{x_1}) \, dx_1 \\
&\leq C\int_{I_{ad}} \Per\left((E_t)_{x_1}\right)^2 \, dx_1 \\
&\leq C\int_{I_{ad}} e^{2t} (\sum_i r_{i,x_1})^2 \, dx_1 \\
&\leq C \eta(\Omega) e^{2t} h. \label{eq: estimate volume omega}
\end{align}

Moreover, we see by Fubini that it holds for all $t \in (0,\infty)$, $\varphi \in C^{\infty}_c(\Omega)$ with $|\varphi|\leq 1$ and $i=2,3$
\begin{align}
\int_{E_t} \partial_i \varphi \, dx &\leq 2\pi\int_{I_{ad}} \sum_i r_{i,x_1}^t \, dx_1 \\
&\leq 2\pi\int_{I_{ad}} e^t \sum_i r_{i,x_1} \, dx_1 \\
&\leq 2\pi e^t h. \label{eq: estimate perimeter omega}
\end{align} 

Similar to the proof of Theorem \ref{theorem: Sobolev approximation}, let $T>0$ and $\eta(\Omega)>0$ be such that $C(S)e^T \eta(\Omega) \leq \frac12 \operatorname{dist}(S,\partial S')$. 
Then define $\omega = E_T \cup \left((I \setminus I_{ad}) \times S\right)$.  
We note that the estimates for $\mathcal{L}^3(\omega)$ and $\sup_{\varphi \in C_c^{\infty}, |\varphi| \leq 1} \int_{\omega} \partial_i \varphi \, dx$, $i=2,3$, follow from \eqref{eq: estimate volume omega}, \eqref{eq: estimate perimeter omega} and the fact that $|I \setminus I_{ad}| \leq C \frac{h}{\eta(\Omega)}$.

Let us now fix $N \in \N$.
Then we apply for each $x_1 \in I_{ad}$ Lemma \ref{lemma: balls Fubini} to the family of balls $(B_{i,x_1}^{t,N})_i$ and the function $f_{x_1} = |\nabla' U(x_1,\cdot,\cdot)|^p \1_{S'}$ to obtain
\[
\int_{I_{ad}} \int_0^{\infty} \sum_{i=1}^N r_{i,x_1}^{t,N} \int_{\partial B_{i,x_1}^{t,N} \cap S'} |\nabla' U|^p \, d\Hm^1 \, dt\, dx_1 \leq C(S) \int_{I \times S'} |\nabla' U|^p \, dx.
\]

Similarly, one obtains
\[
\int_{I_{ad}} \int_0^{\infty} \sum_{i=1}^N r_{i,x_1}^{t,N} \int_{\partial B_{i,x_1}^{t,N} \cap S'} |U|^p \, d\Hm^1 \, dt\, dx_1 \leq C(S) \int_{I \times S'} |U|^p \, dx.
\]

Moreover, we have by the coarea formula applied to cylindrical coordinates
\begin{align} 
\int_{0}^{\infty} \sum_i \Hm^1(J_U \cap \partial_{(2,3)} C_i^{t,N}) \, dt \leq \Hm^2(J_U \setminus \bigcup_{i=1}^N C_i),
\end{align}
where $\partial_{(2,3)} C_i^{t,N}$ denotes the part of the boundary of $C_i^{t,N}$ whose normal is in the $x_2$-$x_3$-plane.
Now find $t_N \in (0,T)$ such that $U$ is in $\SBV^p(\bigcup_{x_1 \in I} \{x_1\} \times \bigcup_i \partial B_{i,x_1}^{t_N,N})$ and
\begin{align} \label{eq: int I_{ad} T nabla w}
\int_{I_{ad}} \sum_i r_{i,x_1}^{t_N,N} \int_{\partial B_{i,x_1}^{t_N,N} \cap S'} |\nabla' U|^p \, d\Hm^1 \, dx_1 &\leq 4\frac{C(S)}T \int_{I \times S'} |\nabla' U|^p \, dx, \\
\label{eq: int I_{ad} T w}
\int_{I_{ad}} \sum_{i=1}^N r_{i,x_1}^{t_N,N} \int_{\partial B_{i,x_1}^{t_N,N} \cap S'} |U|^p \, d\Hm^1 \, dx_1 &\leq 4 \frac{C(S)}T \int_{I \times S'} |U|^p \, dx, \\
\sum_i \Hm^1(J_U \cap \partial_{(2,3)} C_i^{t_N,N}) &\leq \frac4T\Hm^2(J_U \setminus \bigcup_{i=1}^N C_i ). \label{eq: coarea}
\end{align}

By the definition of $\eta(\Omega)$ it follows for all $x_1 \in I_{ad}^N$ that $\partial B_{i,x_1}^{t_N} \subseteq S'$.
Also note that $I_{ad}^N$ is the union of finitely many intervals.

Now, we define $w_N: I \times S \to \R$ as follows.
We set $w_N = 0$ on $(I\setminus I_{ad}^N) \times S$ and $w_N=u$ in $(I_{ad}^N \times S) \setminus E^N_{t_N}$.
In order to define $w_N$ on $E_{t_N}^N \cap (I_{ad}^N\times S)$, we recall that the set $E_{t_N}^N$ consists of finitely many cylinders $C_{i}^{t_N,N}$ of the form $\bigcup_{x_1 \in J} \{x_1\} \times B_{i,x_1}^{t_N,N}$ for an interval $J$. Now fix $x_1 \in I\setminus I_{ad}^N$. 
Then we define for a point $y = (x_1,(1-\theta)x_{x_1,i}^{t_N,N} + \theta z) \in \{x_1\} \times B_{i,x_1}^{t_N,N}$ with $\theta \in (0,1)$ and $z \in \{0\} \times \partial B_{i,x_1}^{t_N,N}$ as in the proof of Theorem \ref{theorem: Sobolev approximation} in polar coordinates 
\[
w_N\left(x_1,(1-\theta)x_{i,x_1}^{t_N,N} + \theta z\right) = (1-\theta) \fint_{\partial B_{i,x_1}^{t_N,N}} U \, d\Hm^1 + \theta U(z).
\]

It follows immediately that $w_N$ is measurable.
By \eqref{eq: int I_{ad} T w} it follows
\begin{equation}
\int_{C_{i}^{t_N,N}} |w_N|^p \, dx \leq 4\frac{C(p)}T \int_J r_{i,x_1}^{t_N,N} \int_{\partial B_{i,x_1}^{t_N,N}} | U|^p \, d\Hm^1 \, dx_1.
\end{equation}

Consequently, we have
\begin{equation}
\int_{I \times S} |w_N|^p \, dx \leq \left( 4\frac{C(p)}T + 1 \right) \int_{I \times S'} |U|^p \, dx \leq C\left( 4\frac{C(p)}T + 1 \right) \int_{I\times S} |u|^p \, dx. \label{eq: estimate wN int}
\end{equation}

In particular, there exists $w \in L^p(I\times S)$ such that (up to a not relabeled subsequence) it holds $w_N \rightharpoonup w$ in $L^p$ satisfying
\[
\int_{I \times S} |w|^p \, dx \leq C \int_{I \times S} |u|^p \, dx.
\]

Moreover, it can be checked that $w_N \in \GSBV^p(I \times S)$.
By \eqref{eq: int I_{ad} T nabla w} and the definition of $w_N$ it holds
\begin{equation}
\int_{I \times S} |\nabla ' w_N|^p \, dx \leq C \int_{\Omega} |\nabla' u|^p \, dx \label{eq: est nabla wN}
\end{equation}

\begin{figure}
\includegraphics[scale=0.8]{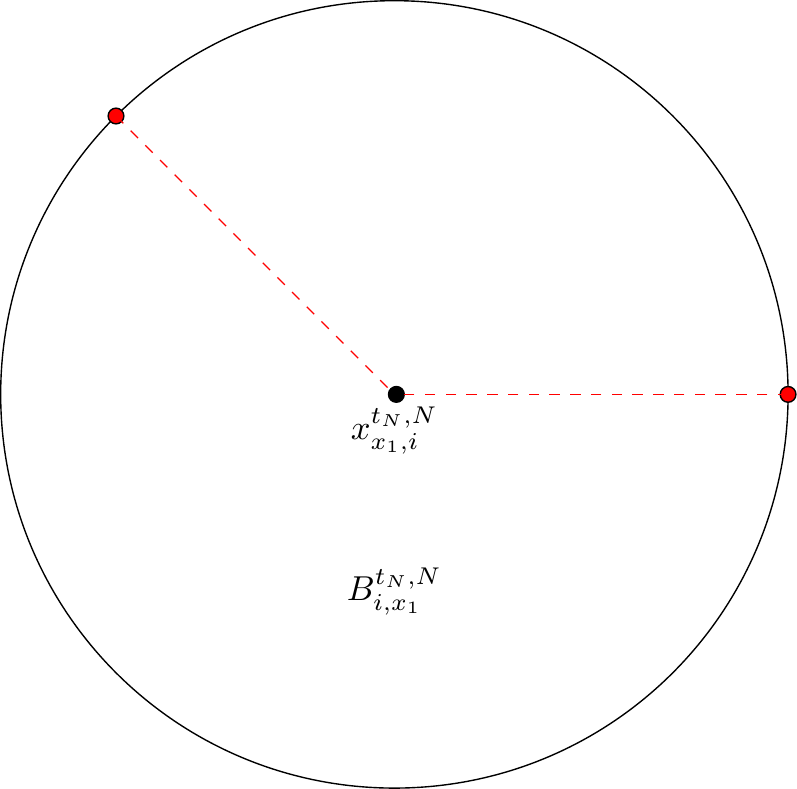}
\caption{Sketch of the jump set of the function $w_N$ in the $x_2$-$x_3$-plane. The red dots on the boundary of the ball $B_{i,x_1}^{t_N,N}$ indicate jumps of the function $U$. By the construction of $w_N$ these jumps are extended into $\{x_1\} \times B_{i,x_1}^{t_N,N}$.}
\label{fig: jump set w_N}
\end{figure}
Moreover, by \eqref{eq: coarea} (see Figure \ref{fig: jump set w_N})
\begin{align}
\int_{J_{w_N}} |\nu'| \, d\Hm^2 &\leq 4 \sum_i r_{i}^{t_N,N} \Hm^1(J_U \cap \partial_{(2,3)} C_i^{t_N,N}) \\
&\leq \frac4T \eta(\Omega) \, \Hm^2(J_U \setminus \bigcup_{i=1}^N C_i). \label{eq: est JwN}
\end{align}

Let us now define the set $A_N = \{ x \in I \times S: \{(x_1,x_2)\} \times \R \cap J_{w_N} \neq \emptyset \}$. 
By estimate \eqref{eq: est JwN} we have that $\mathcal{L}^3(A_N) \leq C \Hm^2(J_U \cap \bigcup_{i=1}^N C_i) \to 0$ as $N\to\infty$.
In particular, $w_N \1_{A_N} \rightharpoonup w$ in $L^p$.
On the other hand, by the definition of $A_N$ we have that $D_2 (w_N \1_{A_N})$ is absolutely continuous with respect to the Lebesgue measure and is by \eqref{eq: est nabla wN} uniformly bounded in $L^p$.
In particular, up to another not relabeled subsequence, there exists $f \in L^p$ such that $D_2(w_N \1_{A_N}) \rightharpoonup f$. 
It follows that $D_2 w = f$, in particular $D_2 w \in L^p$ and 
\[
\int_{I \times S} |D_2 w|^p \, dx \leq C \int_{I\times S} |\nabla' u|^p \, dx.
\]

Similarly one shows $D_3 w \in L^p$ and a corresponding estimate.
It remains to show that $w = u$ outside $\omega$.
By definition we have $w_N = u$ outside $\left(E^N_{t_N} \cap (I_{ad}^N \times S) \right)\cup (I\setminus I_{ad}^N) \times S$.
Note that $\1_{I_{ad}^N} \searrow \1_{I_{ad}}$ as $N\to\infty$.
Moreover, we have by monotonicity $E^N_{t_N} \cap (I_{ad} \times S) \subseteq \omega$.
Hence, $\limsup_{N\to\infty} \1_{\left(E^N_{t_N} \cap (I_{ad}^N \times S) \right)\cup (I\setminus I_{ad}^N) \times S} \leq \1_{\omega}$.
This shows that $w = u$ outside $\omega$.

\end{proof}

\section{Application to an anisotropic Mumford-Shah functional}\label{sec: Mumford Shah}

We now provide an application of Theorem \ref{theorem: Sobolev approximation 3d} to the following Mumford-Shah type functionals $F_\eps : \GSBV^p(I\times S) \to \R$,
\begin{equation}
F_\eps(u) := \int_{I\times S} \frac1\eps |\nabla' u|^p + |\partial_1 u |^p \,dx + \int_{J_u} \frac1\eps |\nu'| + |\nu_1|\,d\Hm^2,
\end{equation}
where $\nu \in S^2$ is the measure theoretic normal to $J_u$.

Here $u$ can be interpreted as a grayscale movie on a film of cross section $S\subset \R^2$, with $x_1\in I$, interpreted as time, and $I\subseteq\R$ is a bounded interval. In this setting, it appears natural to penalize changes in time and changes in space differently.

We show the following result concerning the asymptotics of $F_\eps$ as $\eps \to 0$:

\begin{theorem}\label{theorem: Gamma convergence}
As $\varepsilon \to 0$ the functionals $F_\eps$ $\Gamma$-converge compactly under convergence in measure to the functional $F_0:L^1(I\times S)\to [0,\infty]$,
\begin{equation}
F_0(u) = \begin{cases} \Lm^2(S) \left(\int_I |u'|^p\,dx_1 + \#J_u\right) &\text{ if }u(x) = u(x_1)\in \SBV^p(I)\\
    \infty & \text{ otherwise.}
\end{cases}
\end{equation}

More precisely:

    \begin{itemize}
\item [(i)] Let $u_\eps\in \GSBV^p(I \times S)$ with $\sup_{\eps>0} F_\eps(u_\eps) < \infty$.

Then there exist functions $w_\eps \in \SBV^p_\loc(I)$ such that $u_\eps - w_\eps \to 0$ in measure, with
\begin{equation}
\limsup_{\eps \to 0}   \#J_{w_\eps} \leq  \limsup_{\eps \to 0} \frac1{\Lm^2(S)} \int_{J_{u_\eps}} |\nu_1|\,d\Hm^2
\end{equation}
and
\begin{equation}
\liminf_{\eps \to 0} F_0(w_\eps) \leq \liminf_{\eps \to 0} F_\eps(u_\eps).
\end{equation}

In particular, there exists a sequence of piecewise constant functions $c_{\varepsilon} \in \SBV^p_\loc(I)$, functions $w \in \SBV_{loc}^p(I)$ and a (not relabed) subsequence such that $u_{\varepsilon} - c_{\varepsilon} \to w$  in measure and 
\[
F_0(w) \leq \liminf_{\eps \to 0} F_{\eps} (w_{\eps} - c_{\eps}) \leq \liminf_{\eps \to 0} F_{\eps}(u_{\eps}).
\]

\item [(ii)] For any $u\in \SBV^p(I)$ and any sequence $u_\eps\in \SBV^p(I\times S)$ such that $u_\eps(x) \to u(x_1)$ in measure we have $F_0(u) \leq \liminf_{\eps \to 0} F_\eps(u_\eps)$.

\item [(iii)] For any $u\in \SBV^p(I)$ we have $F_\eps(u) = F_0(u)$.

\end{itemize}
\end{theorem}

Note that a function $u\in GBSV^p(I\times S)$ with $\nabla'u = 0$ and $\int_{J_u} |\nu'| \,d\Hm^2 = 0$ is a function only of $x_1$ almost everywhere, for a proof see for example \cite{AmFuPa}.
Together with the the lower-semicontinuity of all $F_\eps$ this guarantees (ii). Indeed, it holds for all $M > 0$ and $\eps_0 > 0$ that the sequence $v_{\eps} = \max\{ \min\{u_{\eps},M\} u_{\eps}, -M\}$ converges in $L^1$ to $v = \max\{ \min\{u,M\}, -M\}$ and that 
\[
\liminf_{\eps \to 0} F_{\eps}(u_{\eps}) \geq \liminf_{\eps \to 0} F_{\eps_0}(v_{\eps}) \geq F_{\eps_0}(v).
\]

Sending $M \to \infty$ yields $\liminf_{\eps \to 0} F_{\eps}(u_{\eps}) \geq F_{\eps_0}(u)$. Then (ii) follows by sending $\eps_0 \to 0$.
Moreover, (iii) is follows immediately.

Hence, the only difficulty lies in the compactness statement (i), for which no analogues exist in the literature. To show (i), we will have to quantify exactly where a function with bounded $F_\eps$ differs from a function of only $x_1$.

For the proof, the following lemma states that we may remove jumps from a $\SBV^p$ function on an interval if it is close to a different $\SBV^p$ function with fewer jumps.

\begin{lemma}\label{lemma: jump removal}
Let $I\subseteq \R$ be an interval, $p\in (1,\infty)$. Let $u,v\in \SBV^p(I)$ with $\# J_u = N \leq M = \#J_v$. Let $\delta>0$ and $A\subseteq I$ be a Borel set with $\Lm^1(A)\leq \delta$. Then there is a function $w\in \SBV^p(I)$ and a set $\tilde A$ with $\#J_w \leq N$,
\begin{equation}
\int_I |w'|^p\,dx \leq \int_I |v'|^p\,dx,
\end{equation}
\begin{equation}\label{eq: w-v}
    \begin{aligned}
\|w-v\|_{L^\infty(I\setminus \tilde A)} \leq C(M)\bigg( & \|u-v\|_{L^p(I\setminus A)}^{1/p'}(\|u'\|_{L^p} + \|v'\|_{L^p})^{1/p}\\
& + \Lm^1(A)^{1/p'}(\|u'\|_{L^p} + \|v'\|_{L^p}) \bigg),
    \end{aligned}
\end{equation}
and 
\begin{equation}\label{eq: size tilde A}
\Lm^1(\tilde A) \leq C(M)\bigg( \Lm^1(A)+ \frac{\|u-v\|_{L^p(I\setminus A)}}{\|u'\|_{L^p}+\|v'\|_{L^p}} \bigg).
\end{equation}
\end{lemma}
\begin{rem}
Note that if $u$ and $v$ are piecewise constant but $u \neq v$ then $\tilde{A} = I$ is consistent with the estimates of the lemma.
\end{rem}
\begin{proof}
Assume first that $N=0$, by restricting to a smaller interval if necessary. Then $u\in W^{1,p}_\loc(I)$.

Let $J_v = \{t_1,\ldots,t_M\}$ with $t_i \leq t_{i+1}$. Define $t_0 := \inf I$, $t_{M+1} := \sup I$. Define the subintervals $I_i := (t_i,t_{i+1})$.

Let $l \geq 2 \Lm^1(A)$. Let $\tilde A$ be the union of all subintervals of length less than $l$. Then $\Lm^1(\tilde A) \leq (M+1)l$. On all other subintervals $I_i$ we also have $v\in W^{1,p}_\loc(I_i)$.

If $x\in I_i$ and $\Lm^1(I_i) \geq l$, there exists a point $y\in B(x,l) \cap I_i \setminus A$ such that
\begin{equation}\label{eq: u-v}
    |u(y) - v(y)|^p \leq \frac2l \int_{I_i\setminus A} |u-v|^p\,dz.
\end{equation}

Also by the Sobolev empedding theorem both $u$ and $v$ are in $C^{0,1/p'}(I_i)$ and
\begin{equation}\label{eq: differences}
|u(x) - u(y)|^p + |v(x) - v(y)|^p \leq l^{p/p'}\int_{I_i} |u'|^p + |v'|^p\,dz.
\end{equation}

Combining \eqref{eq: u-v} and \eqref{eq: differences} yields
\begin{equation}
\|u-v\|_{L^\infty(I\setminus \tilde A)} \lesssim l^{-1/p} \|u-v\|_{L^p(I\setminus A)} + l^{1/p'} (\|u'\|_{p} + \|v'\|_{p}).
\end{equation}

Optimizing $l$ over $[2\Lm^1(A),\infty)$ yields
\begin{equation}
l := \max \left\{\frac{\|u-v\|_{L^p(I\setminus A)}}{\|u'\|_{L^p} + \|v'\|_{L^p}}, 2\Lm^1(A) \right\},
\end{equation}
whereupon we obtain \eqref{eq: size tilde A} and 
\begin{equation} \label{eq: Linfty estimate}
\|u-v\|_{L^\infty(I \setminus \tilde{A})} \lesssim \max \left\{\|u-v\|_{L^p(I\setminus A)}^{1/p'}(\|u'\|_{L^p} + \|v'\|_{L^p})^{1/p}, \Lm^1(A)^{1/p'}(\|u'\|_{L^p} + \|v'\|_{L^p}) \right\}.
\end{equation}

We now define the function $w$ in $I\setminus \tilde A$ as follows: Pick a point $x_0\in I\setminus \tilde A$ and define
\begin{equation}
w(x) := v(x_0) + \int_{x_0}^x v'(z)\,dz.
\end{equation}

It is clear that $w\in W^{1,p}_\loc(I)$ with $\|w'\|_{L^p} = \|v'\|_{L^p}$. To see that $w$ is close to $v$ in $L^\infty(I\setminus \tilde A)$, we bound the jumps of $v$ between intervals. It is clear that on the interval $I_{i_0}\subset I\setminus \tilde A$ containing $x_0$, we have $w=v$ on $I_{i_0}$. Let $I_i = (t_i,t_{i+1})$ and $I_{i+k} = (t_{i+k},t_{i+k+1})$ be two consecutive intervals contained in $I\setminus \tilde A$. Since $1 \leq k \leq M$ and each interval $I_{i+j}$, $1\leq j \leq k-1$, has length at most $l$, we have $|t_{i+k} - t_{i+1}| \leq Ml$. By the triangle inequality and the H\"older-continuity of $u$ we have
\begin{equation}
|v(t_{i+k}) - v(t_{i+1})| \leq 2 \|u-v\|_{L^\infty(I\setminus \tilde A)} + \|u'\|_{L^p} (Ml)^{1/p'}.
\end{equation}

This happpens at most $M$ times, thus
\begin{equation}
\|w-v\|_{L^\infty(I\setminus \tilde A)} \leq C(M) \|u-v\|_{L^\infty(I\setminus \tilde A)} + C(M) \|u'\|_{L^p} l^{1/p'}.
\end{equation}

Then \eqref{eq: w-v} follows from \eqref{eq: Linfty estimate} and the defition of $l$.

\end{proof}

We now turn to the proof of Theorem \ref{theorem: Gamma convergence} (i). Note that it is impossible to replace the $\limsup$ by a $\liminf$, since different subsequences may minimize different parts of the energy.

\begin{proof}[Proof of Theorem \ref{theorem: Gamma convergence} (i)]

For this proof, we wish to apply Lemma \ref{lemma: jump removal} to $u_\eps$ restricted to two slices $x'=\mathrm{const} \in S$.

Before we continue, let
\begin{equation}
M := \sup_{\eps > 0} \frac{1}{\Lm^2(S)}F_\eps(u_\eps), 
\end{equation}
\begin{equation}
N := \left\lfloor \limsup_{\eps \to 0}\frac{1}{\Lm^2(S)}\int_{J_{u_\eps}} |\nu_1|\,d\Hm^2 \right\rfloor. 
\end{equation}

Extract a subsequence $\eps \to 0$ (not relabeled) such that
\begin{equation}
\lim_{\eps \to 0}\int_{I\times S} |\partial_1 u_\eps|^p\,dx = \liminf_{\eps \to 0} \int_{I\times S} |\partial_1 u_\eps|^p\,dx.
\end{equation}

Then apply Proposition \ref{prop: poincareni} to each function $u_\eps$, yielding sets $\omega_\eps \subset I\times S$ and measurable functions $a_\eps: I \to \R$ such that
\begin{equation}\label{eq: measure omegaeps}
\Lm^3(\omega_\eps) \leq C(p,S,I)M\eps
\end{equation}
and
\begin{equation} \label{eq: diff ueps aeps}
\int_{(I\times S) \setminus \omega_\eps} |u_\eps - a_\eps|^p\,dx \leq C(p,S,I) M \eps.
\end{equation}

We will not need to use the boundary estimate \eqref{eq: est surface exceptional} from Theorem \ref{theorem: Sobolev approximation 3d}.

Next, we define for $\lambda \geq 1$ four subsets of $S$:
\begin{align}
A_\eps^\lambda &:= \left\{ x' \in S\,:\, \int_I |\partial_1 u_\eps(x_1,x')|^p\,dx_1 \leq \frac{\lambda}{\Lm^2(S)} \int_{I\times S} |\partial_1 u_\eps(x)|^p\,dx\right\}, \\
    B_\eps^\lambda &:= \left\{ x' \in S\,:\, u_{\eps}^{x'} = u_{\eps}(\cdot,x') \in SBV^p(I) \text{ and } \#(J_{u^{x'}_\eps}) \leq \frac{\lambda}{\Lm^2(S)} \int_{J_{u_\eps}} |\nu_1|\,d\Hm^2\right\}, \\
    C_\eps^\lambda &:= \left\{ x' \in S\,:\, \int_{I\setminus \omega_\eps^{x'}} |u_\eps(x_1,x') - a_\eps(x_1)|^p\,dx_1 \leq \frac{\lambda}{\Lm^2(S)} \int_{(I\times S)\setminus \omega_\eps} |u_\eps(x) -a_\eps(x_1)|^p \,dx\right\}, \\
    D_\eps^\lambda &:= \left\{ x' \in S\,:\, \Lm^1(\omega_\eps^{x'}) \leq \frac{\lambda}{\Lm^2(S)} \Lm^3(\omega_\eps)\right\},
\end{align}
where, for $x'\in S$, $\omega_\eps^{x'} \subset I$ is the slice
\begin{equation}
\omega_\eps^{x'} := \{x_1\in I\,:\,(x_1,x')\in \omega_\eps\}.
\end{equation}

By Markov's inequality and standard slicing properties of $\BV$-functions we have
\begin{equation}
\min\left\{\Lm^2(A_\eps^\lambda),\Lm^2(B_\eps^\lambda),\Lm^2(C_\eps^\lambda),\Lm^2(D_\eps^\lambda)\right\} \geq \frac{\lambda -1}{\lambda} \Lm^2(S).
\end{equation}

Thus for every $\delta \in (0,1)$, we have
\begin{equation}
\Lm^2(A_\eps^{1+\delta} \cap B_\eps^{8/\delta} \cap C_\eps^{8/\delta} \cap D_\eps^{8/\delta}) \geq \frac{\delta}{8}\Lm^2(S).
\end{equation}

Consequently, we may find $x'\in A_\eps^{1+\delta} \cap B_\eps^{8/\delta} \cap C_\eps^{8/\delta} \cap D_\eps^{8/\delta}$ and define $b_\eps(x_1):= u_\eps(x_1,x')\in \SBV^p(I)$, noting that from the properties of the four sets we have
\begin{equation}
    \begin{aligned}
        \int_I |b_\eps'(x_1)|^p\,dx_1 \leq &\frac{1+\delta}{\Lm^2(S)} \int_{I\times S} |\partial_1 u_\eps(x)|^p\,dx,\\
\#J_{b_\eps} \leq &\frac{8M}{\delta},\\
\int_{I\setminus \omega_\eps^{x'}} |b_\eps(x_1) - a_\eps(x_1)|^p \,dx_1 \leq & \frac{C(p,S,I)M}{\delta}\eps,\\
\Lm^1(\omega_\eps^{x'}) \leq & \frac{C(p,S,I)M}{\delta}\eps.
    \end{aligned}
\end{equation}

Similarly, we can find $y'\in A_\eps^{8/\delta} \cap B_\eps^{1+\delta} \cap C_\eps^{8/\delta} \cap D_\eps^{8/\delta}$ and define $d_\eps(x_1) := u_\eps(x_1,y')\in SBV^p_\loc(I)$, such that
\begin{equation}
    \begin{aligned}
        \int_I |d_\eps'(x_1)|^p\,dx_1 \leq &\frac{C(M,S)}{\delta},\\
\#J_{d_\eps} \leq &\frac{1+\delta}{\Lm^2(S)} \int_{J_{u_\eps}} |\nu_1|\,d\Hm^2,\\
\int_{I\setminus \omega_\eps^{y'}} |d_\eps(x_1) - a_\eps(x_1)|^p \,dx_1 \leq & \frac{C(p,S,I)M}{\delta}\eps,\\
\Lm^1(\omega_\eps^{y'}) \leq & \frac{C(p,S,I)M}{\delta}\eps.
    \end{aligned}
\end{equation}

We note that by our choice of subsequence we have for $\delta\leq \delta_0(N,S)$ and $\eps \leq \eps_0(N,S)$ that $\#J_{d_\eps} \leq N$ uniformly since $\#J_{d_\eps}$ is an integer that is strictly less than $N+1$.
If $\#J_{b_{\eps}} \leq \# J_{d_{\eps}}$ then set $w_{\eps} = b_{\eps}$. It follows imediately that $F_{\eps}(w_{\eps}) \leq (1+\delta) F_{\eps}(u_{\eps})$ and by the estimates \eqref{eq: measure omegaeps}, \eqref{eq: diff ueps aeps} and the definition of $x'$ $w_{\eps} - u_{\eps} \to 0$ in measure. 
In addition, we may assume that $\liminf_{\eps \to 0} \int_I |d_{\eps}'|^p \, dx_1 >0$. Otherwise we may assume up to a further (not relabeled) subsequence that $\lim_{\eps \to 0} \int_I |d_{\eps}'|^p \, dx_1 = 0$. In this scenario, set $w_{\eps} = d_{\eps}$. Then we obtain from \eqref{eq: measure omegaeps}, \eqref{eq: diff ueps aeps} and the definition of $y'$ that $u_{\eps} - w_{\eps} \to 0$ in measure and $\liminf_{\eps \to 0}F_{\eps}(w_{\eps}) \leq (1+\delta) \liminf_{\eps \to 0} F_{\eps}(u_{\eps})$. 

Consequently, from now on, we assume that for $\varepsilon > 0$ small enough it holds $\int_I |d_{\eps}'|^p \, dx_1 >c > 0$ and $\#J_{d_{\eps}} \leq \# J_{b_{\eps}}$. Then apply Lemma \ref{lemma: jump removal} to the pair of functions $b_\eps$, $d_\eps$, with exceptional set $A_\eps := \omega_\eps^{x'} \cup \omega_{\eps}^{y'}$, yielding a new exceptional set $\tilde A_\eps \subset I$ with
\begin{equation}
\Lm^1(\tilde A_\eps) \leq C(M,p,S,c) \frac{\eps^{1/p}}{\delta}
\end{equation}
and a function $w_\eps\in \SBV^p(I)$ such that
\begin{equation}\label{eq: w eps properties}
\begin{aligned}
\int_I |w_\eps'(x_1)|^p\,dx_1 \leq & \frac{1+\delta}{\Lm^2(S)} \int_{I\times S} |\partial_1 u_\eps(x)|^p\,dx,\\
\#J_{w_\eps} \leq& \frac{1+\delta}{\Lm^2(S)} \int_{J_{u_\eps}} |\nu_1|\,d\Hm^2 \leq N,\\
\|w_\eps(x_1) - b_\eps(x_1)\|_{L^\infty(I\setminus \tilde A_{\eps})} \leq & \frac{C(M,p,S)}{\delta} \eps^{1/(pp')}.
\end{aligned}
\end{equation}

The corresponding piecewise constant function $c_{\varepsilon}: I \to \R$ is then simply given as the average of $w_{\varepsilon}$ in between two jumps of $w_{\varepsilon}$.
In particular, by the Poincar\'e inequality, we obtain $L^p$-boundedness of $w_{\varepsilon}-c_{\eps}$.
Hence, there exists $w\in L^p(I)$ such that $w_\eps - c_\eps \rightharpoonup w$ in $L^p$.
Even more, we obtain by the $GSBV^p$ compactness statement that $w \in SBV^p$ and $F_0(w) \leq \liminf_{\eps \to 0} F_0(w_{\eps}) = \liminf_{\eps \to 0} F_{\eps}(w_{\eps}-c_{\eps})$.
Moreover, by \eqref{eq: w eps properties} we have
\begin{equation}\label{eq: est energy}
F_{\varepsilon}(w_{\varepsilon}-c_{\eps}) \leq (1+\delta) F_{\eps}(u_{\eps}).
\end{equation}

Combining \eqref{eq: measure omegaeps}, \eqref{eq: diff ueps aeps} and \eqref{eq: w eps properties} implies for every $\delta_0(N,S) >\delta > 0$ that $w_{\eps} - u_{\eps} \to 0$ in measure and hence $u_\eps - c_\eps \to w$ in measure.
By a diagonal argument, we may ignore the multiplicative error term $1+\delta$ in \eqref{eq: w eps properties} and \eqref{eq: est energy} asymptotically without losing the convergence in measure.

\end{proof}

\bibliographystyle{plain}
\bibliography{poincare_v2}

\end{document}